\documentclass[12pt,reqno,twoside]{amsart}

\usepackage{amssymb,amsmath,amscd,enumerate,verbatim}
\usepackage[pagebackref=true,colorlinks=true,linkcolor=blue,citecolor=blue]{hyperref}

\usepackage[latin1]{inputenc}
\usepackage{color}
\usepackage{graphicx}
\usepackage{tikz}

\usepackage[all]{xy}
\usepackage{pstricks, pst-3d}
\newcommand{\mm}{\mathfrak m}

\newcommand{\kk}{\mathrm k}

\DeclareMathOperator{\pnt}{\raise 0.5mm \hbox{\large\bf.}}

\DeclareMathOperator{\astab}{astab}
\DeclareMathOperator{\diam}{diam}

\DeclareMathOperator{\depth}{depth}

\DeclareMathOperator{\Ass}{Ass}

\DeclareMathOperator{\dist}{dist}

\DeclareMathOperator{\supp}{supp}

\let\phi=\varphi
\let\:=\colon

\newtheorem{thm}{\bf Theorem}[section]
\newtheorem{lem}[thm]{\bf Lemma}

\theoremstyle{definition}
\newtheorem{defn}[thm]{\bf Definition}

\theoremstyle{plain}
\newtheorem*{thm*}{Theorem}
\newtheorem*{lem*}{Lemma}
\newtheorem*{cor*}{Corollary}
\newtheorem*{claim*}{Claim}
\newtheorem*{defn*}{Definition}

\theoremstyle{remark}
\newtheorem{rem}[thm]{Remark}

\numberwithin{equation}{section}
\title{Associated primes of powers of closed neighborhood ideals and diameters of graphs}

\author{Ha Thi Thu Hien}
\address{Foreign Trade University, 91 Chua Lang, Hanoi, Vietnam}
\email{thuhienha504@gmail.com}
\author{Thanh Vu}
\address{Institute of Mathematics, VAST, 18 Hoang Quoc Viet, Hanoi, Vietnam}
\email{vuqthanh@gmail.com}
\thanks{}
\date{\today}

\begin{document}

\begin{abstract} Let $G$ be a simple connected graph and $t \ge 2$ an integer. We prove that if the maximal homogeneous ideal is an associated prime of the $t$th power of the closed neighborhood ideal of $G$, then the diameter of $G$ is at most $7t - 8$. We further show that this bound is sharp for all $t \ge 2$.
\end{abstract}

\maketitle

\section{Introduction} 
Let $G$ be a finite simple graph with vertex set $V(G) = [n] = \{1, \ldots, n\}$ and edge set $E(G)$. For a vertex $i \in V(G)$, the \emph{neighborhood} of $i$ is defined as 
\[
N_G(i) = \{ j \mid \{i, j\} \in E(G) \},
\]
and its \emph{closed neighborhood} is $N_G[i] = N_G(i) \cup \{i\}$. Let $S = \kk[x_1, \ldots, x_n]$ be the standard graded polynomial ring over a field $\kk$. For any subset $W \subseteq [n]$, we denote the squarefree monomial $x_W = \prod_{j \in W} x_j$. The \emph{closed neighborhood ideal} of $G$, denoted by $NI(G)$, is defined by:
\[
NI(G) = ( x_{N_G[i]} \mid i = 1, \ldots, n ).
\]

While the combinatorics of closed neighborhood complexes is a classical subject \cite{B, L}, the algebraic study of closed neighborhood ideals is comparatively recent \cite{CJRS, SM, HS, NBR, NQ, NQBM}. In \cite{HV}, Hien and Vu investigated the associated primes of the second power of $NI(G)$ and  proved that if the maximal homogeneous ideal $\mathfrak{m} = (x_1, \ldots, x_n)$ is an associated prime of $S/NI(G)^2$, then $\operatorname{diam}(G) \le 6$, and this bound is sharp. This result establishes a compelling link between the algebraic properties of neighborhood ideals and the diameter of the underlying graph. In this work, we generalize this connection to arbitrary powers of the ideal. We recall the definition of the diameter of a graph and state our main result.

\begin{defn}
Let $G$ be a simple graph. For vertices $u, v \in V(G)$, the \emph{distance} between $u$ and $v$, denoted by $\operatorname{dist}_G(u, v)$, is the length of the shortest path between them; if no such path exists, then $\operatorname{dist}_G(u, v) = \infty$. The \emph{diameter} of $G$, denoted by $\operatorname{diam}(G)$, is the maximum distance between any two vertices of $G$.
\end{defn}

\begin{thm}\label{thm_diam}
Let $G$ be a simple connected graph and $t \ge 2$ be an integer. Assume that $\mathfrak{m}$ is an associated prime of $S/NI(G)^t$. Then 
\[
\operatorname{diam}(G) \le 7t - 8,
\]
and this bound is sharp.
\end{thm}

We note that, unlike the case $t = 2$, a disconnected graph $G$ may satisfy $\mathfrak{m} \in \operatorname{Ass}(S/NI(G)^t)$ when $t \geq 3$. Indeed, $\mm$ is an associated prime of $S/NI(G)^t$ if and only if $\depth(S/NI(G)^t) = 0$. Assume that $G = G_1 \cup \cdots \cup G_s$, where $G_i$ are the connected components of $G$. It follows from \cite[Theorem 1.1]{NV} that $\mm$ is an associated prime of $S/NI(G)^t$ if and only if there exist positive integers $a_1, \ldots, a_s$ such that $a_1 + \cdots + a_s = t + s - 1$ and $\mm_i$ is an associated prime of $S_i/NI(G_i)^{a_i}$, where $S_i$ and $\mm_i$ are the polynomial ring and the maximal homogeneous ideal on the variables corresponding to the vertices of $G_i$, respectively. Hence, when $t \geq 3$ and $G$ is disconnected, our theorem can be applied to bound the diameters of the connected components of $G$. The idea of bounding the diameter of a graph under the condition that $\mathfrak{m}$ is an associated prime of a power of the closed neighborhood ideal is motivated by a result of Hien, Lam, and Trung \cite[Proposition 2.6]{HLT} on the associated primes of powers of edge ideals of graphs.

The proof technique extends the framework established by Hien and Vu \cite{HV}. In the following section, we provide the necessary notation and prove our main result.

\section{Associated primes of powers of closed neighborhood ideals}
Throughout this paper, we denote by $S = \kk[x_1, \ldots, x_n]$ the polynomial ring over a field $\kk$, and by $\mm = (x_1, \ldots, x_n)$ its maximal homogeneous ideal. For a nonzero monomial $f \in S$ and an index $i$ with $1 \leq i \leq n$, we denote by $\deg_i(f)$ the largest exponent $r$ such that $x_i^r$ divides $f$. The \emph{support} of $f$, denoted by $\supp(f)$, is the set of all variables of $S$ that divide $f$. For an $S$-module $M$, the notation $\Ass(M)$ denotes the set of associated primes of $M$. By \cite[Lemma 2.3]{HV}, we have the following.

\begin{lem}\label{lem_quotient}
Let $J$ be a squarefree monomial ideal. Then $\mm$ is an associated prime of $S/J^t$ if and only if there exists a monomial $f$ such that $\deg_i(f) < t$ for all $i$ and $J : f = \mm$.
\end{lem}

Before proving the main theorem, we construct a graph $G_t$ with $\diam(G_t) = 7t - 8$ and $\mm \in \Ass(S/NI(G_t)^{t+1})$ for all $t \ge 1$. The construction is based on \cite[Example 2.10]{HV}. We recall the graph described there as follows. Let $G = G_1$ be a graph on $12$ vertices labeled $1, \ldots, 12$, with edges 
\begin{align*}
\{1,2\},\{2,3\},\{2,5\},\{3,4\},\{3,8\},\{4,6\},\{4,10\},&\{5,6\},\{5,9\},\\
\{6,7\},\{7,8\},\{7,11\},\{8,9\},&\{9,10\},\{10,11\},\{11,12\}.
\end{align*}
The graph $G_t$ is obtained by taking $t$ copies of $G$ and connecting the leaves of consecutive copies. More precisely,
\[
V(G_t) = \{x_{1,1}, \ldots, x_{1,12}, \ldots, x_{t,1}, \ldots, x_{t,12}\},
\]
where the restriction of $G_t$ to $\{x_{i,1}, \ldots, x_{i,12}\}$ is isomorphic to $G$ for each $i$, and we add edges 
\[
\{x_{1,12}, x_{2,1}\}, \{x_{2,12}, x_{3,1}\}, \ldots, \{x_{t-1,12}, x_{t,1}\}.
\]
The graph $G_2$ is depicted in the following figure.

\begin{figure}
\centering
\begin{tikzpicture}[
    scale=1.2,
    vertex/.style={circle,draw,fill=white,inner sep=0.5pt,minimum size=18pt,font=\scriptsize},
    edge/.style={line width=0.8pt}
]

\def\rin{1.2}
\def\rout{2.2}
\def\rleaf{2.8}

% --- FIRST COPY (reflected labels) ---
\begin{scope}[shift={(-3.5,0)}]
    \foreach \i/\ang in {6/90, 5/162, 9/-126, 8/-54, 7/18}{
      \node[vertex] (x\i) at (\ang:\rin) {$x_{\i}$};
    }
    \node[vertex] (x4)  at (270:\rout) {$x_{4}$};
    \node[vertex] (x11) at (20:\rout)  {$x_{11}$};
    \node[vertex] (x10) at (-50:\rout) {$x_{10}$};
    \node[vertex] (x2)  at (160:\rout) {$x_{2}$};
    \node[vertex] (x3)  at (-130:\rout) {$x_{3}$};
    \node[vertex] (x12) at (20:\rleaf)  {$x_{12}$};
    \node[vertex] (x1)  at (160:\rleaf) {$x_{1}$};

    \draw[edge] (x12)--(x11);
    \draw[edge] (x11)--(x10);
    \draw[edge] (x11)--(x7);
    \draw[edge] (x10)--(x9);
    \draw[edge] (x10)--(x4);
    \draw[edge] (x9)--(x8);
    \draw[edge] (x9)--(x5);
    \draw[edge] (x8)--(x7);
    \draw[edge] (x8)--(x3);
    \draw[edge] (x7)--(x6);
    \draw[edge] (x6)--(x5);
    \draw[edge] (x6)--(x4);
    \draw[edge] (x5)--(x2);
    \draw[edge] (x4)--(x3);
    \draw[edge] (x3)--(x2);
    \draw[edge] (x2)--(x1);
\end{scope}

% --- SECOND COPY (reflected labels) ---
\begin{scope}[shift={(2.5,0)}]
    \foreach \i/\ang in {6/90, 5/162, 9/-126, 8/-54, 7/18}{
      \node[vertex] (y\i) at (\ang:\rin) {$y_{\i}$};
    }
    \node[vertex] (y4)  at (270:\rout) {$y_{4}$};
    \node[vertex] (y11) at (20:\rout)  {$y_{11}$};
    \node[vertex] (y10) at (-50:\rout) {$y_{10}$};
    \node[vertex] (y2)  at (160:\rout) {$y_{2}$};
    \node[vertex] (y3)  at (-130:\rout) {$y_{3}$};
    \node[vertex] (y12) at (20:\rleaf)  {$y_{12}$};
    \node[vertex] (y1)  at (160:\rleaf) {$y_{1}$};

    \draw[edge] (y12)--(y11);
    \draw[edge] (y11)--(y10);
    \draw[edge] (y11)--(y7);
    \draw[edge] (y10)--(y9);
    \draw[edge] (y10)--(y4);
    \draw[edge] (y9)--(y8);
    \draw[edge] (y9)--(y5);
    \draw[edge] (y8)--(y7);
    \draw[edge] (y8)--(y3);
    \draw[edge] (y7)--(y6);
    \draw[edge] (y6)--(y5);
    \draw[edge] (y6)--(y4);
    \draw[edge] (y5)--(y2);
    \draw[edge] (y4)--(y3);
    \draw[edge] (y3)--(y2);
    \draw[edge] (y2)--(y1);
\end{scope}

% --- BRIDGE ---
\draw[edge, red] (x12) -- (y1);

\end{tikzpicture}
\caption{A graph achieving the maximal diameter bound}
\label{fig:graph_max_diameter}
\end{figure}
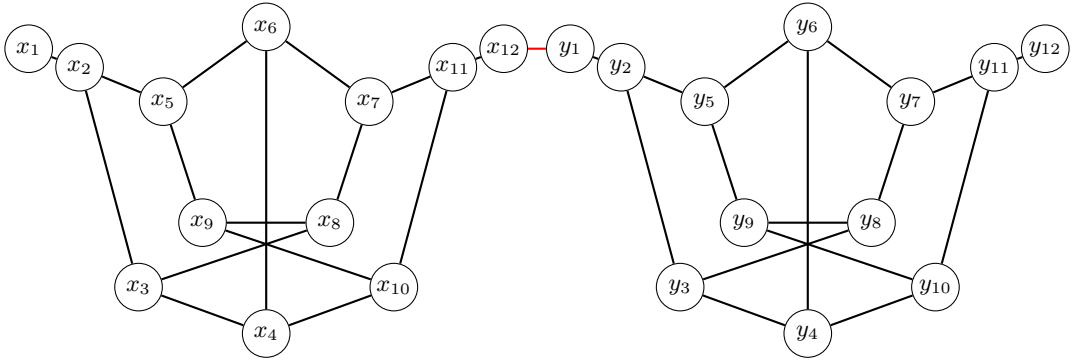

We have the following lemma.

\begin{lem}\label{lem_extreme_graph}
We have $\diam(G_t) = 7t - 1$ and $\mm \in \Ass(S/NI(G_t)^{t+1})$.
\end{lem}

\begin{proof} Note that $\diam (G_1) = 6$. By construction, each edge $\{x_{i,12}, x_{i+1,1}\}$ is a bridge for $i = 1,\ldots, t-1$. Hence,
\[
\dist_{G_t}(x_{1,1}, x_{t,12}) = 6t + (t - 1) = 7t - 1,
\]
which is the maximum possible distance. Therefore, $\diam(G_t) = 7t - 1$.

Let $f_i = x_{i,2} \cdots x_{i,11}$ and $f = f_1 \cdots f_t$. Note that $N_{G_t}[\supp(f_i)] \cap N_{G_t}[\supp(f_j)] = \emptyset$ for all $i \neq j$. Since $f_i \in NI(G_t)$ but $f_i \notin NI(G_t)^2$, we deduce that $f \in NI(G_t)^t$ but $f \notin NI(G_t)^{t+1}$. It suffices to show that $x_{i,j} \in NI(G_t)^{t+1} : f$ for all $i = 1, \ldots, t$ and $j = 1, \ldots, 12$.

It is enough to verify that $x_{i,j} f \in NI(G_t)^{t+1}$, or equivalently, that $x_{i,j} f_i \in NI(G_t)^2$ for each copy. The case $t=1$ follows from \cite[Example 2.10]{HV}, so we assume that $t \ge 2$. 

For each $x_{i,1}$, we need to find two vertices $x_{i,k}$ and $x_{i,l}$ such that $N_{G_t}[x_{i,k}] \cap N_{G_t}[x_{i,l}] = \emptyset$ and 
\[
N_{G_t}[x_{i,k}] \cup N_{G_t}[x_{i,l}] \subseteq \{x_{i,1}\} \cup \supp(f_i).
\]
Indeed, for $x_{1,1}$, the vertices $x_{1,1}$ and $x_{1,9}$ satisfy the condition. For $x_{i,1}$ with $i = 2, \ldots, s$, the vertices $x_{i,2}$ and $x_{i,10}$ satisfy the condition. The vertices $x_{i,12}$ for $i = 1, \ldots, t$ can be handled similarly by symmetry. 

Now, for $x_{i,j}$ with $j = 2, \ldots, 11$, we need to show that there exist two vertices $x_{i,k}$ and $x_{i,l}$ such that 
\[
N[x_{i,k}] \cap N[x_{i,l}] = \{x_{i,j}\} \quad \text{and} \quad N[x_{i,k}] \cup N[x_{i,l}] \subseteq \supp(f).
\]
We observe that the following pairs of vertices satisfy this condition for $x_{i,j}$ with $j = 2, \ldots, 11$:
\[
\begin{aligned}
&j=2: \ x_{i,3}, x_{i,5}, \quad 
&j=3: \ x_{i,4}, x_{i,8}, \quad 
&j=4: \ x_{i,3}, x_{i,10}, \\
&j=5: \ x_{i,6}, x_{i,9}, \quad 
&j=6: \ x_{i,5}, x_{i,7}, \quad 
&j=7: \ x_{i,6}, x_{i,8}, \\
&j=8: \ x_{i,7}, x_{i,9}, \quad 
&j=9: \ x_{i,5}, x_{i,8}, \quad 
&j=10: \ x_{i,4}, x_{i,9}, \\
&j=11: \ x_{i,7}, x_{i,10}.
\end{aligned}
\]
The conclusion follows.
\end{proof}

\begin{rem}
The closed neighborhood ideal $NI(G_t)$ differs from the sum of the closed neighborhood ideals of each copy of $G_1$ in $G_t$. This difference arises from the connections at the endpoints of these copies.
\end{rem}

We now present a simple preparatory lemma and fix some notation for the proof of the main theorem.

\begin{lem}\label{lem_disjoint_ind}
Let $G$ be a connected graph with $\diam(G) = d$. Then there exist $\left\lfloor \frac{d}{3} \right\rfloor + 1$ vertices of $G$ whose closed neighborhoods are pairwise disjoint.
\end{lem}

\begin{proof}
Let $v_1, \ldots, v_{d+1}$ be an induced path of maximum length in $G$. Let $a =  \lfloor \frac{d}{3} \rfloor$. Then the vertices $v_1, v_4, \ldots, v_{3a+1}$ satisfy $N_G[v_i] \cap N_G[v_j] = \emptyset$ for all $i \neq j$.
\end{proof}

\begin{defn}
Let $G$ be a simple graph and let $U, V \subseteq V(G)$.
\begin{enumerate}
    \item The \emph{closed neighborhood} of $U$ in $G$ is defined by   \(
    N_G[U] = \bigcup_{u \in U} N_G[u].
    \)
    \item The \emph{induced subgraph} of $G$ on $U$, denoted by $G[U]$, is the graph with vertex set $U$ and edge set $E(G) \cap (U \times U)$.
    \item The \emph{distance} between $U$ and $V$ is defined by
    \(
    \dist_G(U, V) = \min \{ \dist_G(u, v) \mid u \in U,\, v \in V \}.
    \)
\end{enumerate}
\end{defn}

In the following proof, we sometimes use $N[v]$ instead of $N_G[v]$ and $N[U]$ instead of $N_G[U]$ for simplicity of notation.
\begin{proof}[Proof of Theorem \ref{thm_diam}]
By Lemma \ref{lem_extreme_graph}, it suffices to show that if $\mm \in \Ass(S/NI(G)^t)$, then $\diam(G) \leq 7t - 8$. We therefore assume that $\mm \in \Ass(S/NI(G)^t)$. By Lemma \ref{lem_quotient}, there exists a monomial $f$ such that $\mm = NI(G)^t : f$. In particular, $f \notin NI(G)^t$. Let $U = \supp(f)$ and 
\[
W = \{u \in U \mid N_G[u] \subseteq U\}.
\]
For ease of reading, we divide the proof into several steps.

\medskip

\noindent \textbf{Step 1.} $W$ is nonempty and $N_G[W] = U$.

Clearly, $f \neq 1$, hence $U \neq \emptyset$. Let $u \in U$. Since $x_u f \in NI(G)^t$, there exist vertices $i_1, \ldots, i_t$ such that $x_{N[i_1]} \cdots x_{N[i_t]}$ divides $x_u f$. In particular, $N[i_j] \subseteq \supp(uf) = \supp(f) = U$ for all $j = 1, \ldots, t$. Hence, $i_1, \ldots, i_t \in W$, and therefore $W$ is nonempty.

Furthermore, since $f \notin NI(G)^t$, there exists some $j$ such that $u \in N[i_j]$. Thus, $u \in N_G[i_j] \subseteq N_G[W]$, and hence $N_G[W] = U$.

\medskip

\noindent \textbf{Step 2.} $N_G[U] = V(G)$.

If $U = V(G)$, there is nothing to prove. Let $v \notin U$. Then $x_v f \in NI(G)^t$, so there exist vertices $v_1, \ldots, v_t$ such that $x_{N[v_1]} \cdots x_{N[v_t]}$ divides $x_v f$. Note that this product does not divide $f$, so there exists some $j$ such that $v \in N[v_j]$.

Since $v \notin U$, we have $\deg_v(x_v f) = 1$. Hence,
\[
N_G[v_j] \setminus \{v\} \subseteq \supp(f) = U.
\]
If $v_j \neq v$, then $v_j \in U$. Hence, $v \in N_G[v_j] \subseteq N_G[U]$. If $v_j = v$, then there exists $u_j \in N_G[v] \setminus \{v\}$. Then $u_j \in U$, which also implies that $v \in N_G[u_j] \subseteq N_G[U]$. Thus, $N_G[U] = V(G)$.

\medskip

\noindent \textbf{Step 3.} Assume that $G[W]$ has $s$ connected components $W_1, \ldots, W_s$. Let $U_i = N_G[W_i]$ and $V_i = N_G[U_i]$. Then $\dist_G(V_i, V_j) \leq s - 1$ for all $i, j = 1, \ldots, s$, where by convention $\dist_G(U, \emptyset) = 0$ for any subset $U \subseteq V(G)$.

We may assume that $s \geq 2$. Since $G$ is connected, the sets $V_i$ are pairwise connected through the graph. As there are $s$ components, it follows that $\dist_G(V_i, V_j) \leq s - 1$ for all $i, j$.

\medskip

\noindent \textbf{Step 4.} Final step.

Let $\diam(W_i) = a_i$ for $i = 1, \ldots, s$. By Step 1, $\diam(U_i) \leq \diam(W_i) + 2$. By Step 2, $\diam(V_i) \leq \diam(U_i) + 2 \leq \diam(W_i) + 4$. By Step 3,
\[
\diam(G) \leq \sum_{i=1}^s \diam(W_i) + 4s + (s - 1).
\]

Assume that $\diam(W_i) = a_i$. By Lemma \ref{lem_disjoint_ind}, we can choose $\left\lfloor \frac{a_i}{3} \right\rfloor + 1$ vertices in $W_i$ whose closed neighborhoods are pairwise disjoint. Since $f \notin NI(G)^t$, there can be at most $t - 1$ such vertices in total. Therefore,
\[
\sum_{i=1}^s \left(\left\lfloor \frac{a_i}{3} \right\rfloor + 1\right) \leq t - 1.
\]

Under this constraint, the sum $\sum_{i=1}^s a_i$ is maximized when each $a_i$ is of the form $a_i = 3b_i + 2$. The condition then becomes $\sum_{i=1}^s b_i \leq t - s - 1$. Hence,
\[
\diam(G) \leq \sum_{i=1}^s a_i + 5s - 1 \leq 3(t - s - 1) + 2s + 5s - 1 = 3t + 4s - 4.
\]
Since $s \leq t - 1$, we conclude that $\diam(G) \leq 7t - 8$.
\end{proof}

\begin{rem}
By a result of Brodmann \cite{Br}, there exists an integer $t_0$ such that $\Ass(S/I^t) = \Ass(S/I^{t_0})$ for all $t \ge t_0$. The smallest such integer is called the stability index of the associated primes of powers of $I$, denoted by $\astab(I)$. 

With the notation as in Lemma~\ref{lem_extreme_graph} and Theorem~\ref{thm_diam}, we deduce that $\mm \in \Ass(S/NI(G_t)^{t+1})$ but $\mm \notin \Ass(S/NI(G_t)^t)$. In particular, $\astab(NI(G_t)) \ge t$. Hence, this provides a family of graphs $G$ for which $\astab(NI(G))$ can be arbitrarily large.
\end{rem}

We conclude the paper with a remark on an analogous result concerning bounds on the diameter of graphs whose maximal homogeneous ideal is an associated prime of powers of their edge ideals; this serves as a motivation for our work. Note that this result can be deduced from the work of Hien, Lam, and Trung \cite{HLT} and Lam and Trung \cite{LT}. Recall that the edge ideal of $G$, denoted by $I(G)$, is generated by all monomials $x_i x_j$ such that $\{i,j\}$ is an edge of $G$.

\begin{thm}
Let $G$ be a connected simple graph and let $t \ge 2$ be an integer. Assume that $\mm$ is an associated prime of $S/I(G)^t$. Then
\[
\diam(G) \le 4t - 5,
\]
and the bound is sharp.
\end{thm}

\begin{proof}
Let $f = x_1^{a_1} \cdots x_n^{a_n}$ be such that $I(G)^t : f = \mm$. Let $U = \supp(f)$. Assume that $U_1, \ldots, U_s$ are the connected components of $G[U]$ with $|U_i| = n_i$. Let $V_i = N_G[U_i]$. Then $\diam(V_i) \le \diam(U_i) + 1$. 

By \cite[Corollary 2.2]{HLT}, $N_G[U] = V(G)$; hence $V_1 \cup \cdots \cup V_s = V(G)$. Since $G$ is connected, $\dist_G(V_i, V_j) \le s - 1$. By \cite[Lemma 2.8 and Lemma 2.9]{HLT}, each $U_i$ contains an odd cycle. In particular, $\diam(U_i) \le n_i - 2$. By \cite[Proposition 2.6]{HLT}, we have $\sum_{i=1}^s n_i \le 3(t - 1)$. Hence,
\begin{align*}
\diam(G) 
&\le \sum_{i=1}^s (\diam(U_i) + 2) + s - 1 \\
&\le \sum_{i=1}^s n_i + s - 1 \le 3t + s - 4.
\end{align*}
Since $f \notin I(G)^t$, we deduce that $s \le t - 1$. Hence, $\diam(G) \le 4t - 5$.

We now construct a graph $G_t$ attaining the maximal bound. Let $G_2$ be the graph on $5$ vertices $\{1,2,3,4,5\}$ with edges $\{1,2\}, \{2,3\}, \{3,4\}, \{2,4\}, \{4,5\}$. In other words, $G_2$ is a triangle with two leaves attached to two vertices of the triangle. Then $G_t$ is obtained by connecting $(t-1)$ copies of $G_2$ at their endpoints. Namely, $G_t$ is the graph with vertex set
\[
V(G_t) = \{x_{i,j} \mid i = 1, \ldots, t-1,\ j = 1, \ldots, 5\},
\]
such that the induced subgraph of $G_t$ on $\{x_{i,1}, \ldots, x_{i,5}\}$ is isomorphic to $G_2$, and there is an edge between $x_{i,5}$ and $x_{i+1,1}$ for $i = 1, \ldots, t - 2$. Then one can easily check that $\diam(G_t) = 4t - 5$ and $\mm \in \Ass(S/I(G_t)^{t})$. This completes the proof.
\end{proof}

\vspace{1mm}
\subsection*{Data availability}

Data sharing is not applicable to this article, as no datasets were generated or analyzed during the current study.

\subsection*{Conflict of interest}

The authors have no relevant financial interests to disclose.

\subsection*{Acknowledgments}
Ha Thi Thu Hien is partially supported by Foreign Trade University under research program number FTURP02-2026-13.

\end{document}